 \documentclass[letterpaper,11 pt]{article}

\usepackage{amsmath} 
\usepackage{amssymb}  
\usepackage{acronym,wrapfig,plain,mathrsfs,enumerate,relsize,color}
\usepackage[colorlinks=true,citecolor=blue]{hyperref}
\definecolor{blue}{RGB}{0,0,255} 

\providecommand{\keywords}[1]{\textbf{\textit{keywords:}} #1}
\allowdisplaybreaks
\usepackage{graphicx}
\usepackage{spverbatim}
\newtheorem{algorithm}{Algorithm}
\usepackage{algorithm}
\usepackage{amsmath}
\usepackage{multirow}
\usepackage{algorithmic}
\usepackage{subfig}
\usepackage{url}            
\usepackage{booktabs}       
\usepackage{amsfonts}       
\usepackage{nicefrac}       
\usepackage{microtype}      
\usepackage{xcolor}         
\usepackage{graphicx}
 
\usepackage{amsfonts,amsmath,fullpage,bbm}
\usepackage{amssymb,amsthm,multirow,verbatim}
\usepackage{acronym,wrapfig,plain,mathrsfs,enumerate,relsize,color}
\usepackage{multirow}
\newtheorem{theorem}{Theorem}

\newtheorem{lemma}{Lemma}

\newtheorem{remark}{Remark}

\newtheorem{assumption}{Assumption}
\newtheorem{definition}{Definition}
\newtheorem{example}{Example}
\def\ml{\mathcal L}

\newcommand{\mor}[1]{{\color{black}#1}}

\newcommand{\zal}[1]{{\color{black}#1}}
\newcommand{\af}[1]{{\color{black}#1}}
\newcommand{\zizi}[1]{{\color{black}#1}}
\newcommand{\mort}[1]{{\color{black}#1}}
\newcommand{\afj}[1]{{\color{black}#1}}
\newcommand{\mb}[1]{{\color{black}#1}}

\newcommand{\mbrs}[1]{{\color{black}#1}}
\newcommand{\z}[1]{{\color{black}#1}}
\newcommand{\mrt}[1]{{\color{black}#1}}
\newcommand{\az}[1]{{\color{black}#1}}

\allowdisplaybreaks

\title{\LARGE \bf
Accelerated Primal-dual Scheme for a Class of Stochastic Nonconvex-concave Saddle Point Problems
}

\author{
  \begin{tabular}{c}
    Morteza Boroun\textsuperscript{*} \qquad Zeinab Alizadeh \textsuperscript{*} \qquad Afrooz Jalilzadeh\footnote{Department of Systems and Industrial Engineering, The University of Arizona, Tucson, AZ, USA.\\
     \texttt{\{morteza, zalizadeh, afrooz\}@arizona.edu}}
  \end{tabular}
}
\date{}

\onecolumn
\begin{document}

\maketitle
\thispagestyle{empty}
\pagestyle{empty}

\begin{abstract}
    \normalsize

Stochastic nonconvex-concave min-max saddle point problems appear in many machine learning and control problems including distributionally robust optimization, generative adversarial networks, and adversarial learning. In this paper, we consider a class of nonconvex saddle point problems where the objective function
satisfies the Polyak-Łojasiewicz condition with respect to the minimization variable and it is concave with respect to the maximization variable. The existing methods for solving nonconvex-concave saddle point problems often suffer from slow convergence and/or contain multiple loops. Our main contribution lies in proposing a novel single-loop accelerated primal-dual algorithm with new convergence rate results appearing for the first time in the literature, to the best of our knowledge. \mrt{In particular, in the stochastic regime, we demonstrate a convergence rate of $\mathcal O(\epsilon^{-4})$ to find an $\epsilon$-gap solution which can be improved to $\mathcal O(\epsilon^{-2})$ in deterministic setting.}
\end{abstract}

\section{Introduction}
\label{sec:intro}
In this paper, we consider the following min-max saddle point (SP) game:
\begin{align}\label{main}
& \min_{x\in \mathcal X} \max_{y\in \mathcal Y} {\Phi(x,y)}\triangleq \ml(x,y)-h(y),
\end{align}
where $\mathcal X=\mathbb R^n$, $\mathcal Y=\mathbb R^m$, $\ml(x,y)=\mathbb E[\ml(x,y;\xi)]$, $\xi$ is a random vector, $\mathcal L(\cdot,y)$ is potentially nonconvex for any $y\in \mathcal Y$ and satisfies Polyak-Łojasiewicz (PL) condition (see Definition \ref{def PL}), $\ml(x,\cdot)$ is concave for any $x\in \mathcal X$ and $h(\cdot)$ is convex and possibly nonsmooth. 
Our goal is to develop an algorithm to find a first order stationary point of this SP problem.

Recent emerging applications
in machine learning and control have further stimulated a surge of interest in these problems. Examples that can be formulated as \eqref{main} include generative adversarial
networks (GANs) \cite{goodfellow2016deep},  fair  classification \cite{nouiehed2019solving}, communications \cite{akhtar2021conservative,bedi2019asynchronous}, and wireless system \cite{chen2011convergence,feijer2009krasovskii}. 
Convex-concave saddle point problems have been extensively studied in the literature \cite{chambolle2016ergodic,hamedani2021primal}. However, recent applications in machine learning and control may involve nonconvexity. One class of nonconvex-concave min-max problems is when the objective function satisfies PL condition that we aim to study in this paper. Next, we provide two examples that can be formulated as problem \eqref{main} and satisfies PL condition. 
\begin{example}[Generative adversarial imitation learning]\label{ex2}\emph{
One practical example of PL-game is generative adversarial imitation learning of linear quadratic regulators (LQR). Imitation learning techniques aim to mimic human behavior by observing
 an expert demonstrating a given task  \cite{hussein2017imitation}. Generative adversarial imitation learning (GAIL) is studied in  \cite{ho2016generative} which solves imitation learning via min-max optimization. 
Let $K$ represents the choice of the policy, $K_E$ represents the expert policy, and the cost parameter and the expected cumulative cost for a given policy $K$ are denoted by $\theta=(Q,R)$ and $C(K,\theta)$, respectively. The problem of GAIL for LQR can be formulated \cite{cai2019global} as $
\min_{K} \max_{\theta\in \Theta} m(K,\theta),$
where $m(K,\theta)=C(K,\theta)-C(K_E,\theta)$, $Q\in \mathbb R^{d\times d }$, $R\in \mathbb R^{k\times k}$, and $\Theta\triangleq \{(Q,R)\mid \alpha_Q I \preceq Q \preceq \beta_Q I, \alpha_R I \preceq R \preceq \beta_R I\}$. 
It is known that $m$ satisfies PL condition in $K$ \cite{nouiehed2019solving}. This problem is a special case of \eqref{main}, for $h(\theta)=\mathbb I_{\Theta}(\theta)$, where $\mathbb I_{\Theta}$ denotes the indicator function of set $\Theta$.}
\end{example}
\begin{example}[Distributionally robust optimization]\label{ex1}\emph{
 Define $\ell_i(x)=\ell(x,\xi_i)$, where $\ell:\mathcal X\times \Omega\to\mathbb R$ is a loss function possibly nonconvex and $\Omega=\{\xi_1,\hdots,\xi_n\}$. Distributionally robust optimization (DRO) studies worse case performance under uncertainty to find solutions with some specific confidence level \cite{namkoong2016stochastic}. DRO can be formulated as $\min_{x\in \mathcal X}\max_{y\in  Y} \sum_{i=1}^n y_i \ell_i(x),$
where $\mathcal Y$ represents the uncertainty set, e.g., $ Y=\{y\in \mathbb R^m_{+}\mid y\geq \delta/n, \ V(y,\tfrac{1}{n}\mathbf 1_n)\leq \rho\}$ is an uncertainty set considered in \cite{namkoong2016stochastic} and $V (Q,P)$ denotes the divergence measure between two sets of probability measures $Q$ and $P$. As it has been shown in \cite{guo2020fast}, DRO for deep learning with ReLU activation function satisfies PL condition in an $\epsilon$-neighborhood around a random initialized point. This problem is a special case of \eqref{main}, for $h(y)=\mathbb I_{Y}(y)$.}
\end{example}
One natural way to solve problem \eqref{main} is directly with the idea of taking two simultaneous or sequential steps for reducing the objective function $\Phi(\cdot,y)$ for a given $y$ and increasing the objective function $\Phi(x,\cdot)$ for a given $x$. One of the most famous algorithms for solving such problem is known as gradient descent-ascent (GDA) \cite{nedic2009subgradient}.
It has been discovered that such a naive approach leads to poor performance and may even diverge for  simple problems.
One way to resolve this issue is by adding a momentum in terms of the gradient of the objective function. Although this approach leads to an optimal convergence rate result \cite{hamedani2021primal,zhao2021accelerated}, it may not be directly applicable in nonconvex-concave setting. Therefore, we aim to to develop a novel primal-dual algorithm with acceleration in the primal update as well as a new momentum in the dual update. 
\subsection{Related Works}
{\bf Nonconvex-concave SP problem.} Various algorithms have been proposed for solving nonconvex-concave SP problems due to their applicability in many modern machine learning problems. The existing methods can be categorized into two types: multi-loop and single-loop. In multi-loop algorithms \cite{kong2021accelerated,ostrovskii2021efficient} one variable is updated in a few consecutive iterations until a certain condition is satisfied before another variable gets updated. Such methods are often difficult to implement in practice because the termination of the inner loop has a high impact on the overall complexity of such algorithms, and selecting a conservative criterion may lead to a high computational cost while an inadequate number of inner iterations may lead to  poor performance. Therefore, there have been some recent efforts \cite{lu2020hybrid,zhang2020single,xu2020unified} to design and analyze single-loop algorithms to solve nonconvex-concave problems. In particular, a convergence rate of $\mathcal O(\epsilon^{-4})$ has been obtained for the aforementioned single-loop algorithms. Authors in \cite{zhang2020single} were able to improve the rate to $\mathcal O(\epsilon^{-2})$ for a special case of nonconvex-concave problem, i.e., $\min_x\max_{y\in Y}f(x)^Ty$, where $Y$ is a probability simplex.
There are also several studies \cite{rafique2018weakly,lin2020gradient,zhang2022sapd} in the stochastic regime. See Table \ref{result table} for more details. 

{\bf PL condition.} Rate results for nonconvex-concave problems can be improved for a class of problems where the objective function satisfies PL condition. Recently, nonconvex-PL SP problems have been studied in \cite{nouiehed2019solving,anagnostidis2021direct} and \cite{yang2022faster} assuming that the objective satisfies one-sided PL condition. Multi-loop algorithms \cite{nouiehed2019solving,anagnostidis2021direct} find an $\epsilon$–first order stationary
point of the problem within \mb{$\mathcal{\tilde O} (\epsilon^{-2})$} iterations, where $\tilde O(\cdot)$  denotes $\mathcal O(\cdot)$ up to a logarithmic factor. The same rate result has been achieved in \cite{fiez2021global} and \cite{yang2022faster} for a single-loop schemes.  More recently, to guarantee a global convergence, Yang et al. \cite{yang2020global} proposed alternating gradient descent ascent algorithm with a linear convergence rate to solve SP problem where the objective satisfies two-sided PL condition. Moreover, the convergence rate of $\mathcal O(\epsilon^{-1})$ has been shown for the stochastic regime under two-sided PL condition. Subsequently, Guo et al. \cite{guo2020fast} improved the dependency of convergence rate on the condition number (the ratio of smoothness parameter to the PL constant).
\begin{table}[htb]
\vspace{-3mm}
\caption{Comparison of complexity between some of the main existing methods for solving SP problem}
\label{result table}
\renewcommand{\arraystretch}{2}
\centering
\resizebox{0.5\linewidth}{!}{
\centering
\begin{tabular}{|c|c|cc|c|}
\hline
\multirow{2}{*}{References} & \multirow{2}{*}{Problem} & \multicolumn{2}{c|}{Complexity}                 & \multirow{2}{*}{\# of loops} \\ \cline{3-4}
                            &                          & \multicolumn{1}{c|}{det.} & stoch. &                             \\ \hline
\cite{hamedani2021primal,zhao2021accelerated}&SC-C&                            \multicolumn{1}{c|}{$\mathcal O(\epsilon^{-0.5})$}            & $\mathcal O(\epsilon^{-1})$         & Single                  \\ \hline
\cite{chambolle2016ergodic,juditsky2011solving,zhao2021accelerated}&C-C&                            \multicolumn{1}{c|}{$\mathcal O(\epsilon^{-1})$}            & $\mathcal O(\epsilon^{-2})$         & Single                  \\ \hline
                            
\cite{rafique2018weakly}                          & NC-C                         & \multicolumn{1}{c|}{$\mathcal O(\epsilon^{-6})$}            & $\tilde{\mathcal O}(\epsilon^{-6})$         & Double                  \\ \hline
\cite{lin2020gradient}                           & NC-C                        & \multicolumn{1}{c|}{$\mathcal O(\epsilon^{-6})$}             & $\mathcal O(\epsilon^{-8})$          & Single                  \\ \hline
\cite{zhang2022sapd}&NC-C& \multicolumn{1}{c|}{--}&$\mathcal O(\epsilon^{-6})$&Double\\ \hline
\cite{fiez2021global}                           & NC-PL     & \multicolumn{1}{c|}
{$\mathcal{\tilde O} (\epsilon^{-2})$}             & --          & Single                  \\ \hline
{This paper}                   &  {PL-C}     & \multicolumn{1}{c|}{{$\mathcal O(\epsilon^{-2})$}}             & {$\mathcal O(\epsilon^{-4})$}          & {Single}                        \\\hline
\end{tabular}
}
\end{table}

\subsection{Contributions}
The existing methods for solving nonconvex-concave SP problems often suffer from slow convergence and/or contain multiple loops. Our main contribution lies in proposing a novel single-loop accelerated primal-dual algorithm with convergence rate results for PL-game appearing for the first time in the literature to the best of our knowledge. 
Our main contributions are summarized as follows:
(i) We propose an accelerated primal-dual scheme to solve problem \eqref{main}. Our main idea lies in designing a novel algorithm by combining an accelerated step in the primal variable with a dual step involving a momentum in terms of the gradient of the objective function. 
(ii) Under a stochastic setting, using an acceleration where mini-batch sample gradients are utilized, our method achieves an oracle complexity (number of sample gradients calls) of $\mathcal O(\epsilon^{-4})$. (iii) Under a deterministic regime,   
    we demonstrate a convergence guarantee of $\mathcal O(\epsilon^{-2})$ to find an $\epsilon$-stationary solution. This is the best-known rate for SP problems satisfying one-sided PL condition to the best of our knowledge. 

\section{Preliminaries}
First we define some important notations. 

{\bf Notations.} $\|x\|$ denotes the Euclidean vector norm, i.e., $\|x\|=\sqrt{x^Tx}$. $\mbox{prox}_g(x)$ denotes the proximal operator with respect to $g$ at $x$, i.e., $\mbox{prox}_g(y)\triangleq \mbox{argmin}_x\{\tfrac{1}{2}\|x-y\|^2+g(x)\}$. $\mathbb E[x]$ is used to denote the expectation of a random variable $x$. We define $x^*(y)\triangleq \mbox{argmin}_x \ml({x,y)}$. Given the mini-batch samples $ \mathcal U=\{\xi^i\}_{i=1}^b$ and $\mathcal V=\{\bar \xi^i\}_{i=1}^b$, we let $\nabla_x \ml_{\mathcal U}(x,y)={1\over b} \sum_{i=1}^b\nabla_x \ml(x,y;\xi^i)$ and $\nabla_y \ml_{\mathcal V}(x,y)={1\over b} \sum_{i=1}^b\nabla_y \ml(x,y;\bar \xi^i)$. We defined $\sigma$-algebras $\mathcal H_k=\{\mathcal U_1,\mathcal V_1,\mathcal U_2,\mathcal V_2,\hdots,\mathcal U_{k-1}, \mathcal V_{k-1}\}$ and  $\mathcal F_k=\{\mathcal H_k\cup V_k\}$.

Now we briefly highlight a few aspects of the PL condition \cite{polyak1963gradient} that differentiate it from convexity and make it a more relevant and appealing setting for many machine learning applications. 
For unconstrained minimization problem $\min_{x\in \mathbb R^n} f(x)$, we say that a function satisfies the PL inequality if  for some $\mu>0$,
${1\over 2}\|\zizi{\nabla } f(x)\|^2\geq \mu(f(x)-f(x^*))$ for all $x\in \mathbb R^n$.
 To verify the PL condition, we need access to the value of the objective function the norm of the gradient which is often tractable and can be estimated from a sub-sample data. However, for verifying convexity, one needs to estimate the minimum eigenvalue of the Hessian matrix. Moreover, the norm of the gradient is much more resilient to perturbation of the objective function than the smallest eigenvalue of the Hessian \cite{bassily2018exponential}.

PL condition does not require strong convexity or even convexity of the objective function. 
It has been shown that it is satisfied for different class of problems, for instance, conditions like restricted secant inequality \cite{zhang2013gradient} and one-point convexity \cite{allen2018natasha} are special cases of PL condition. Problems satisfying such conditions include dictionary learning \cite{arora2015simple}, neural networks \cite{li2017convergence} and phase \mort{retrieval} \cite{chen2015solving}, to name a few. 
In this paper, we consider a min-max SP problem and we assume that the objective function satisfies one-sided PL inequality.
\begin{definition}\label{def PL}
A continuously differentiable function $\ml(x,y)$ satisfies the one-sided PL condition if there exists a constant $\mu>0$ such that ${1\over2}\|\nabla_x \ml{(x,y)}\|^2\geq \mu(\ml{(x,y)}- \ml({x^*(y),y)}),$ for all $x\in \mathcal X, y\in \mathcal Y,$
where $\ml({x^*(y),y)})=\min_x \ml({x,y)}$.
\end{definition}


Now we state our main assumptions.
 \begin{assumption}\label{assump0} (i) The solution set of problem \eqref{main} is nonempty; (ii)
Function $h(y)$ is convex and possibly nonsmooth; (iii) $\ml(x,y)$ is continuously differentiable satisfying one-sided PL condition and $\ml(x,\cdot)$ is concave for any $x\in \mathcal X$.
\end{assumption}


\begin{assumption}\label{assump1} $\nabla_x\ml$ is Lipschitz continuous, i.e., there exist $L_{xx}\geq0$ and $L_{xy}\geq0$ such that
$ \|\nabla_x \ml(x,y)- \nabla_x \ml(\bar x,\bar y)\| \leq{L_{xx}} \|x-\bar x\|+{L_{xy}} \|y-\bar y\|.$
Moreover, $\ml(x,y)$ is linear in terms of $y$.
 \end{assumption}
 Note that Assumption \ref{assump1} implies that 
 \begin{align}\label{assump3}
 \ml(x,y)-\ml(\bar x,y)-\langle \nabla_x \ml(\bar x,y),x-\bar x\rangle \leq \tfrac{L_{xx}}{2} \mor{\|x-\bar x\|^2}.
 \end{align}
 
 Under stochastic setting, we assume that the sample gradients can be generated \zizi{by} satisfying the following standard conditions.
 
 \begin{assumption}\label{assump:stoch}
 Each component function $\ml(x,y;\xi)$ has unbiased stochastic gradients with bounded variance:
\begin{align*}\quad &\mathbb E[\nabla_x\ml(x,y,\xi)\mid \mathcal F_k]=0,\ \mathbb E[\nabla_y\ml(x,y,\xi)\mid \mathcal H_k]=0,\\
&\mathbb E[\|\nabla_x\ml(x,y;\xi)-\nabla_x\ml(x,y)\|^2]\leq \nu^2_x,\\
&\mathbb E[\|\nabla_y\ml(x,y;\xi)-\nabla_y\ml(x,y)\|^2]\leq \nu^2_y.
\end{align*}
 \end{assumption}

\section{Primal-Dual Method with Momentum}\label{sec:alg}
In this section, we propose a primal-dual algorithm with momentum (PDM) for deterministic PL-concave problems. The details of the method can be seen in Algorithm \ref{alg2}. Then, we introduce stochastic PDM (SPDM) for stochastic setting (see Algorithm \ref{alg1}).

\begin{algorithm}[htb]
\caption{ Primal-Dual with Momentum (PDM)}
 \label{alg2}
\begin{algorithmic}[1]
   \STATE Given $x_0, \tilde x_0,y_0$, $\alpha_k\in(0,1]$, and positive sequences $\{\sigma_k\}$, $\{\gamma_k$\} and $\{\lambda_k\}$;
   \FOR{$k=0\hdots T-1$}
   \STATE\label{update z2} $z_{k+1} =(1-\alpha_k)\tilde x_{k} +\alpha_kx_{k}$;
   \STATE \label{update p,q2} $p_k=\nabla_y \ml{(z_{k+1},y_{k})}$ and $q_k={1\over \gamma_k(1-L_{xx}\gamma_k) \mu}(\nabla_y\ml{(x_{k},y_{k})}-\nabla_y\ml{(x_{k-1},y_{k})})$;
   \STATE\label{update y2} $y_{k+1}= \mbox{prox}_{\sigma_{k},h}\left(y_{k}+\sigma_k (p_k+ q_k)\right)$;
   \STATE\label{update r2} $r_k=\nabla_x \ml(z_{k+1},y_{k+1})$;
   \STATE\label{update x2} $x_{k+1}= \left(x_k -\gamma_k r_k\right)$;
   \STATE\label{update tilde x2} $\tilde x_{k+1}=(z_{k+1}-\lambda_k r_k)$;
   \ENDFOR
\end{algorithmic}
\end{algorithm}
Algorithm \ref{alg2} consists of a single loop primal-dual steps. After initialization of parameters, 
at each iteration $k\geq 0$, a proximal gradient  ascent step for the variable $y$ is taken in the direction of $\nabla_y \ml$ with an additive momentum term $q_k$. Such a momentum is an algorithmic approach to gain acceleration for solving PL-concave problems 
Finally, after computing gradient $\nabla_x\ml$ at $(z_{k+1},y_{k+1})$, two gradient descent steps for the variable $x$ is taken to generate $x_{k+1}$ and $\tilde x_{k+1}$ which then will be combined by a convex combination in the next iteration. 

\begin{remark}
If we let $\lambda_k=\alpha_k\zizi{\gamma_k}$, then the primal step in Algorithm \ref{alg2} will be similar to one of the variants of the Nesterov’s acceleration (see \cite{nesterov2003introductory} and \cite{ghadimi2016accelerated}). Moreover, when $\lambda_k=\gamma_k$ it can be shown that $z_{k+1}=x_k$ and $x_{k+1}=\tilde x_{k+1}$ which is similar to a gradient descent step for the minimization variable. 
\end{remark}

For a stochastic setting, SPDM is proposed in  Algorithm \ref{alg1} where the main steps of the algorithm is similar to  Algorithm \ref{alg2}. The main difference is that  instead of computing the exact gradient, we estimate the gradient of the function by drawing mini-batch samples $\mathcal U_k$ and $\mathcal V_k$ in Step \ref{gen sample}.
\begin{algorithm}[htb]
\caption{Stochastic Primal-Dual with Momentum (SPDM)}
 \label{alg1}
\begin{algorithmic}[1]
   \STATE Given $x_0, \tilde x_0,y_0$, $\alpha_k\in(0,1]$, and positive sequences $\{\sigma_k\}$, $\{\gamma_k$\} and $\{\lambda_k\}$;
   \FOR{$k=0\hdots T-1$}
   \STATE\label{update z} $z_{k+1} =(1-\alpha_k)\tilde x_{k} +\alpha_kx_{k}$;
   \STATE\label{gen sample} Generate randomly mini-batch samples \\$ \mathcal U_k=\{\xi^i_k\}_{i=1}^b$ and $\mathcal V_k=\{\bar \xi^i_k\}_{i=1}^b$; 
   \STATE \label{update p,q} $q_k=\tfrac{1}{(\gamma_k-L_{xx}\gamma_k^2) \mu}(\nabla_y\ml_{\mathcal V_k}{(x_{k},y_{k})}-\nabla_y\ml_{\mathcal V_k}{(x_{k-1},y_{k})})$ and $p_k=\nabla_y \ml_{\mathcal V_k}{(z_{k+1},y_{k})}$;
   \STATE\label{update y} $y_{k+1}= \mbox{prox}_{\sigma_{k},h}\left(y_{k}+\sigma_k (p_k+ q_k)\right)$;
   \STATE\label{update r} $r_k=\nabla_x \ml_{\mathcal U_k}(z_{k+1},y_{k+1})$;
   \STATE\label{update x} $x_{k+1}= \left(x_k -\gamma_k r_k\right)$;
   \STATE\label{update tilde x} $\tilde x_{k+1}=(z_{k+1}-\lambda_k r_k)$;
   \ENDFOR
\end{algorithmic}
\end{algorithm}

 
\section{Convergence Analysis}\label{sec:conv}
In this section, we study the convergence properties of  and \ref{alg1} for stochastic (and also deterministic) settings. All related proofs are provided in the appendix. Our goal is to find a first order stationary point of problem \eqref{main}. For a given positive $\epsilon$, we define a point $(x,y)$ as an $\epsilon$-stationary solution of problem \eqref{main} if $\|\nabla_x\Phi(x,y)\|\leq \epsilon$ and $\nabla_y\mathcal L(x,y)\in h(y)+\mathcal B(0,r\epsilon)$ for some $r>0$. 



For our analysis, for all $k\in(0,T-1)$,  define $C_k$ as:
\begin{align}\label{align1}
&C_k\triangleq 1-L_{xx}\gamma_k-\tfrac{L_{xx}(\gamma_k-\lambda_k)^2}{2\alpha_k\Gamma_k\gamma_k}\left(\sum_{\tau=k}^{T-1} \Gamma_\tau\right)\geq 0,
\end{align}
where $ \Gamma_k \triangleq
    \begin{cases}
      1 & k=0\\
      (1-\alpha_k)\Gamma_{k-1} & k\geq1\\
    \end{cases}.$
\begin{remark}
By choosing $\alpha_k=\tfrac{2}{k+1}$, $\lambda_k=\tfrac{1}{2L_{xx}}$ and $\gamma_k\in [\lambda_k,(1+\alpha_k/4)\lambda_k]$ for any $k\geq 0$, from definition of $C_k$, one can show that $C_k\geq 11/32$ (see \cite{ghadimi2016accelerated}).
\end{remark}
Now we establish the convergence rate of SPDM for solving stochastic PL-concave SP problem \eqref{main}. In Algorithm \ref{alg1}, to estimate the gradient of the function, we draw mini-batch samples $\mathcal U_k$ and $\mathcal V_k$ at each iteration, where $|\mathcal U_k|=|\mathcal V_k|=b$. 
\begin{theorem}\label{th1}
Let $\{x_k, y_{k},z_k\}_{{k} \geq0}$ generated by Algorithm \ref{alg1} and suppose Assumptions \ref{assump0}, \ref{assump1} and \ref{assump:stoch} hold.
 Moreover, let $\sigma_k= \tfrac{\mu}{36L^2_{xy}}$ $\alpha_k=\tfrac{2}{k+1}$, $\lambda_k=\tfrac{1}{2L_{xx}}$ and $\gamma_k\in [\lambda_k,(1+\alpha_k/4)\lambda_k]$ for any $k\geq 0$ and $b=T$. Then,  there exists an iteration $k\in\{0,\hdots,T\}$ such that $(z_k,y_k)$ is an $\epsilon$-stationary point of problem \eqref{main} which can be obtained within $\mathcal O(\epsilon^{-4})$ evaluations of sample gradients. 
\end{theorem}


Consider function $\ml$ in problem \eqref{main} to be deterministic, i.e. exact gradients $\nabla_x \ml$ and $\nabla_y \ml$ are available. We show that the convergence rate can be improved to $\mathcal O(\epsilon^{-2})$. 
\begin{theorem}\label{th2}
Let $\{x_k, y_{k},z_k\}_{{k} \geq0}$ generated by Algorithm \ref{alg2} and suppose Assumptions \ref{assump0}, \ref{assump1} hold.
 Choosing parameters as Theorem \ref{th1}, there exists an iteration $k\in\{0,\hdots,T\}$ such that $(z_k,y_k)$ is an $\epsilon$-stationary point which can be obtained within $\mathcal O(\epsilon^{-2})$ evaluations of the gradients. 
\end{theorem}
The proof for deterministic setting, i.e, Theorem \ref{th2}, is similar to Theorem \ref{th1}, by letting $\nu_x=\nu_y=0$.
\section{Numerical Results}\label{sec:numeric}
{\bf Generative Adversial Imitation Learning.} In this section, we implement our method to solve GAIL problem described in Example \ref{ex2}. The code utilized in our experiment was adapted from an existing implementation developed by \cite{yang2020global}. 
To validate the efficiency of the proposed scheme, \afj{we compare PDM algorithm with alternating gradient descent ascent (AGDA) \cite{yang2020global}, Smoothed-GDA \cite{zhang2020single}, and AGP \cite{xu2020unified}.} 
The optimal control problem for LQR can be formulated as follows \cite{cai2019global}:  
\begin{align}\label{num:problem}
&\underset{\pi_t}{\text{minimize} }\quad \mathbb E\left[\sum_{t=0}^{\infty}x_t^{\top}Qx_t + u_t^{\top}Ru_t\right]\\
\nonumber&\text{subject to}\ x_{t+1}=Ax_t+Bu_t,\ u_t=\pi_t(x_t), \ 
x_0\sim \mathbb D_0,
\end{align}
where $Q\in \mathbb R^{d\times d }$, $R\in \mathbb R^{k\times k}$ are both positive definite matrices,  $A\in \mathbb R^{d\times d }$, $B\in \mathbb R^{d\times k }$, , $u_t\in \mathbb R^{k}$ is a control, $x_t\in \mathbb R^{d}$ is a state, $\pi_t$ is a policy, and $\mathbb D_0$ is a given initial distribution. In the infinite-horizon setting with a stochastic initial state $x_0\sim \mathbb D_0$, the optimal control input can be written as a linear function $u_t=-K^*x_t$ where $K^*\in \mathbb R ^{k\times d}$ is the  policy and does not depend on $t$. We denote the expected cumulative cost in \eqref{num:problem} by $C(K,\theta)$, where $\theta=(Q,R)$. To estimate the expected cumulative cost, we sample $n$ initial points $x_0^{(1)},\hdots,x_0^{(n)}$ and estimate $C(K,\theta)$ using sample average:
$C_{n}(K;\theta):={1\over n} \sum_{i=1}^{n} \left[\sum_{t=0}^{\infty}x_t^{\top}Qx_t + u_t^{\top}Ru_t\right]_{x_0=x_0^{(i)}}.$

In GAIL for LQR, the goal is to learn the cost function parameters $Q$ and $R$ from the expert after the trajectories induced by an expert policy $K_E$ are observed. Hence, the min-max formulation of the imitation learning problem is $\min_K\max_{\theta\in \Theta}\ m_n(K,\theta),$
where $m_n(K,\theta)=C_n(K,\theta)-C_n(K_E,\theta)-\phi(\theta)$, $\phi$ is a regularization term that we added so that the problem becomes strongly concave, so can apply AGDA scheme (see \cite{yang2020global}). Moreover, $\Theta$ is the feasible set of the cost parameters.  We assume $\Theta$ is convex and there exist positive constants $\alpha_Q,\beta_Q,\alpha_R$ and $\beta_R$ such that for any $(Q,R)\in \Theta$ we have
$\alpha_QI\preceq Q\preceq \beta_QI, \ \alpha_RI\preceq R\preceq\beta_RI.$
We generate three different data sets for different choices of $d$ and $k$ and we set $n=100$, $\alpha_Q=\alpha_R=0.1$ and $\beta_Q=\beta_R=100$. We choose $\alpha_k=\tfrac{2}{(k+1)}$, $\sigma_k=0.4$ and $\lambda_k=\gamma_k=2e$-4. The exact gradient of the problem in compact form has been established in \cite{fazel2018global}. 
non-accelerated scheme (AGDA). 
\begin{figure}[htb]
    \centering
    \subfloat[d=10, k=7]{{\includegraphics[width=3.8cm]{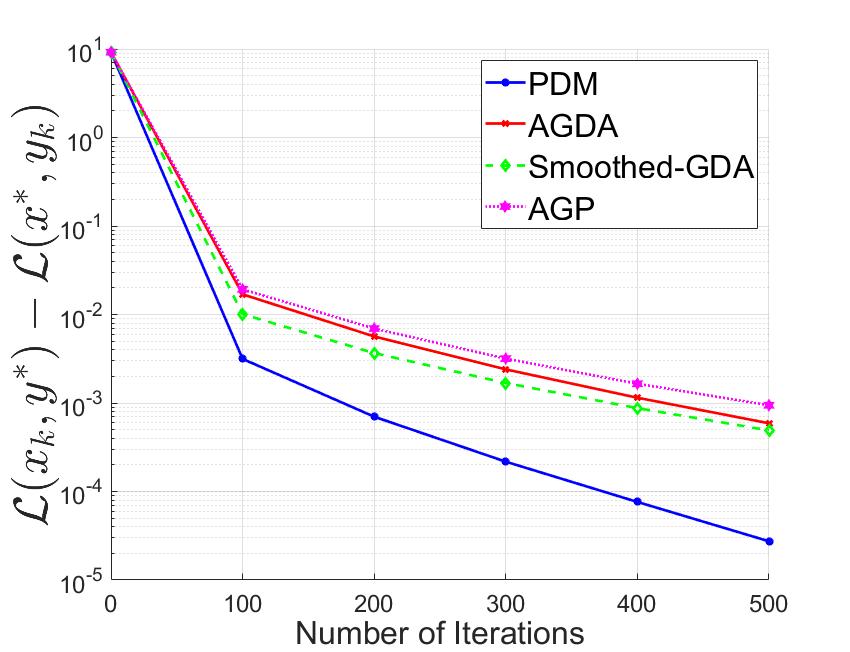} }}
    \subfloat[d=35, k=25]{{\includegraphics[width=3.8cm]{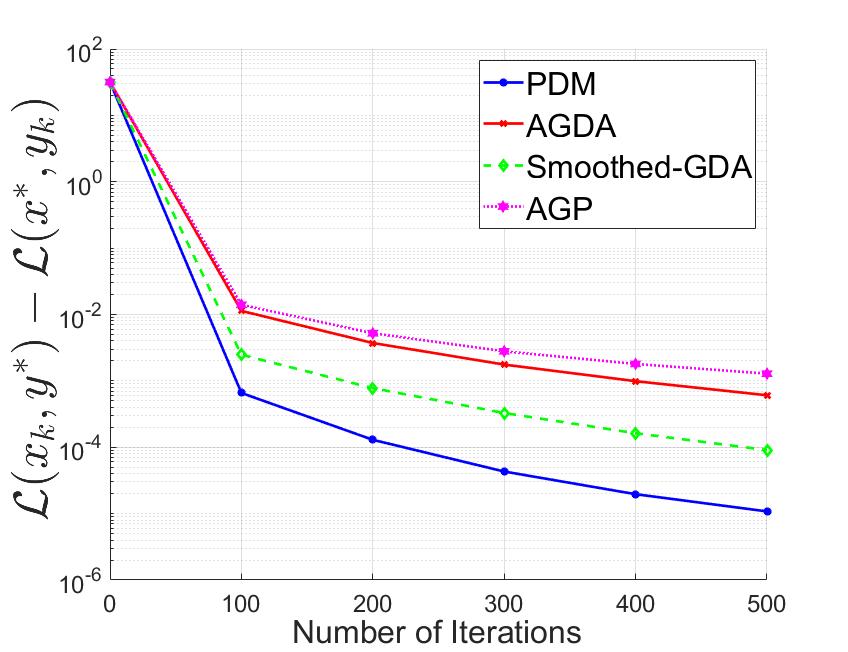} }}
     \subfloat[n=200]{{\includegraphics[width=3.9cm]{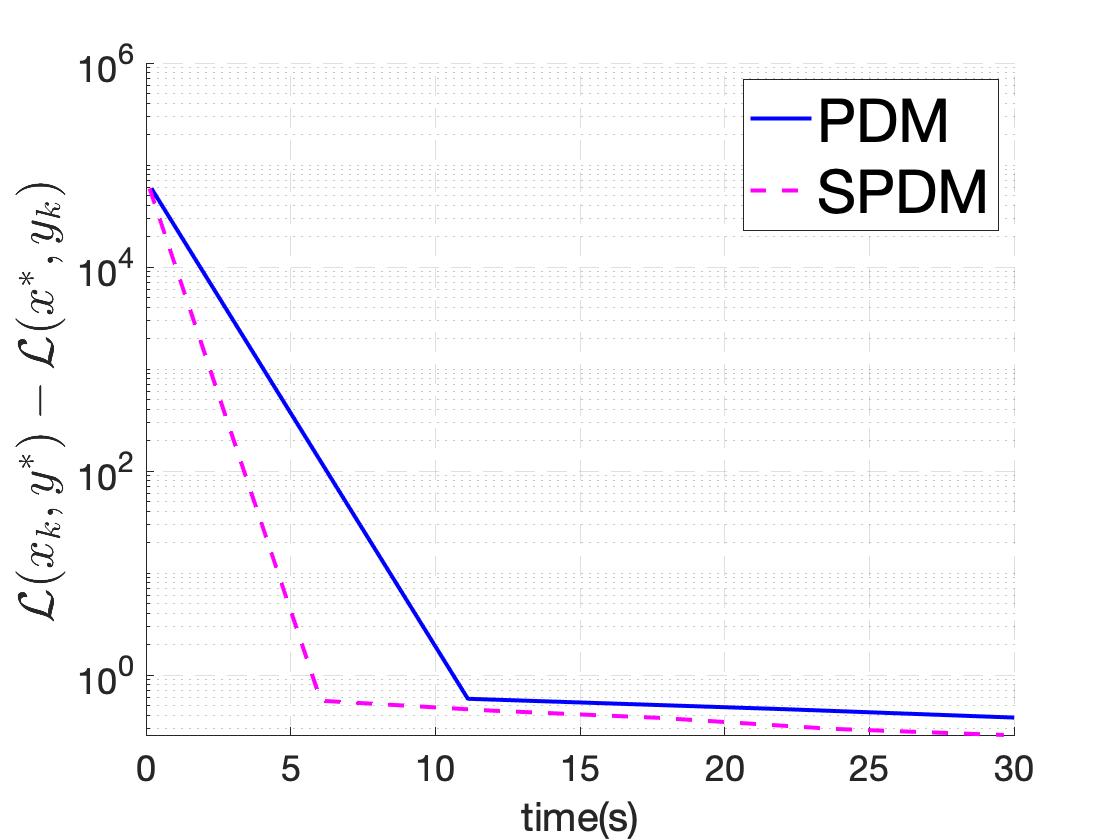} }}
    \subfloat[n=300]{{\includegraphics[width=3.9cm]{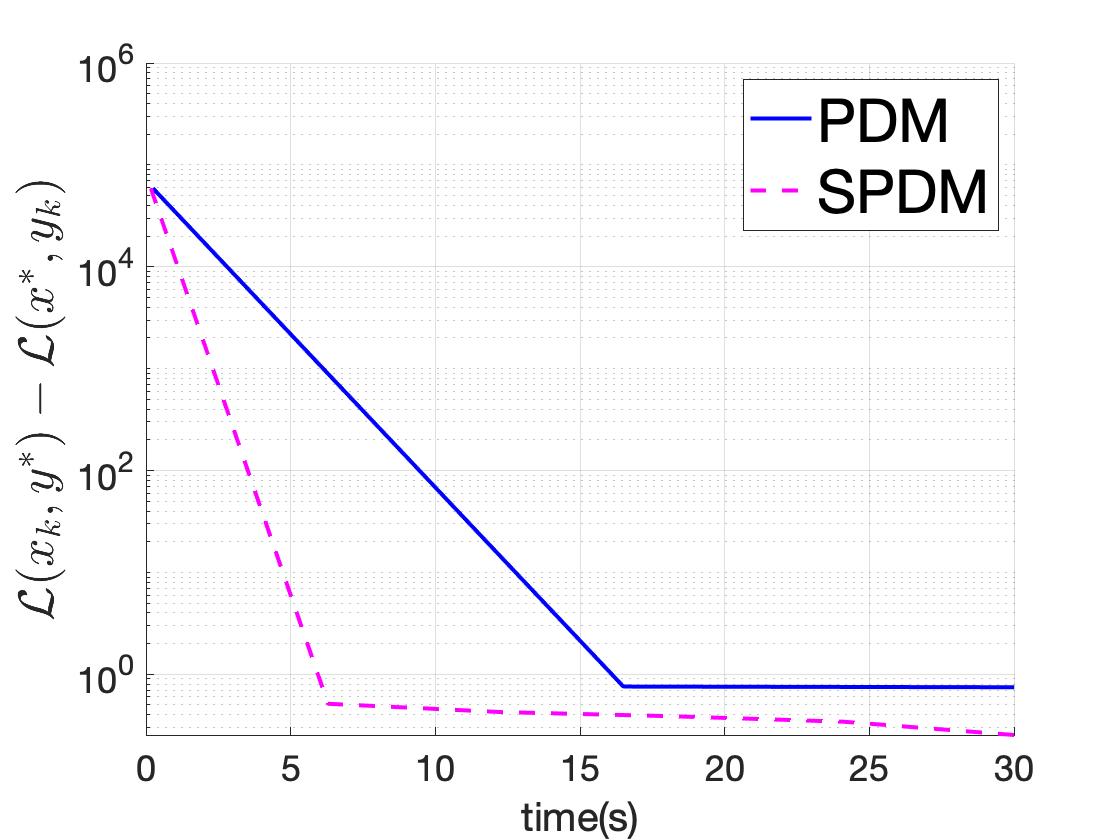} }}
   \caption{ \afj{(a),(b): PDM vs competitive schemes for different dimensions ($n=100$); (c),(d): PDM vs SPDM }}
    \label{figcomp}\vspace{-3mm}
\end{figure}

In Figure \ref{figcomp} (a) and (b), we compared the performance of our proposed method (PDM) with AGDA \cite{yang2020global}, \afj{Smoothed-GDA \cite{zhang2020single}, and AGP \cite{xu2020unified}. We set the same stepsizes for all the methods to ensure fairness in our experiment. Other parameters for competitive methods are selected as suggested in their papers.} 
In Figure \ref{figcomp} (c) and (d), we compared PDM with its stochastic variant (SPDM) by running both algorithms for the same amount of time. As it can be seen SPDM outperforms PDM and its superiority is more evident as $n$ becomes larger.\\
\mrt{{\bf Distributionally robust optimization.}
Consider the following DRO problem.
     \begin{align*}
        \min_{x\in \mathcal X}\max_{y\in \mathcal Y} \sum_{i=1}^n y_i \log (1+\exp(-b_i a_{i}^T x)) ,
     \end{align*}
     Where $\mathcal Y = \{y\in \mathbb R^m_{+}\mid y\geq \delta/n, \ \tfrac{1}{2}\|ny-\mathbf 1_n\|\leq \rho\} $, $\delta = 1/100 $ and $\rho = 50$. We compare our method with stochastic accelerated primal-dual method proposed in \cite{zhao2021accelerated} (SPDHG) and stochastic mirror prox \cite{juditsky2011solving} (SMP). We use real datasets colon-cancer (n=62, m=2000) and leukemia (n=38, m=7129) from  LIBSVM library \cite{libsvm}. Note that in these datasets the number of features are larger than the number of samples, therefore, computing $\nabla_y \mathcal L$ is cheap while $\nabla_x \mathcal L$ can be costly, hence, we use an unbiased estimator $\nabla_x \mathcal L_{\mathcal U}$ with batch size of 10 for all the methods. We run all algorithms for 300 seconds. The performance of the methods are depicted in Figure \ref{fig:dro1}. Table \ref{tabel_dro} summarizes the performance of our algorithm and competitive methods in terms of the gap function. Our scheme outperforms other algorithms which matches with the theoretical result. In fact, PDM has convergence rate of $\mathcal O(1/k)$ and the other two methods have a convergence rate of $\mathcal O(1/\sqrt k)$.}
\begin{figure}[htb]
    \centering
    \subfloat[Colon-cancer]{{\includegraphics[width=3.8cm]{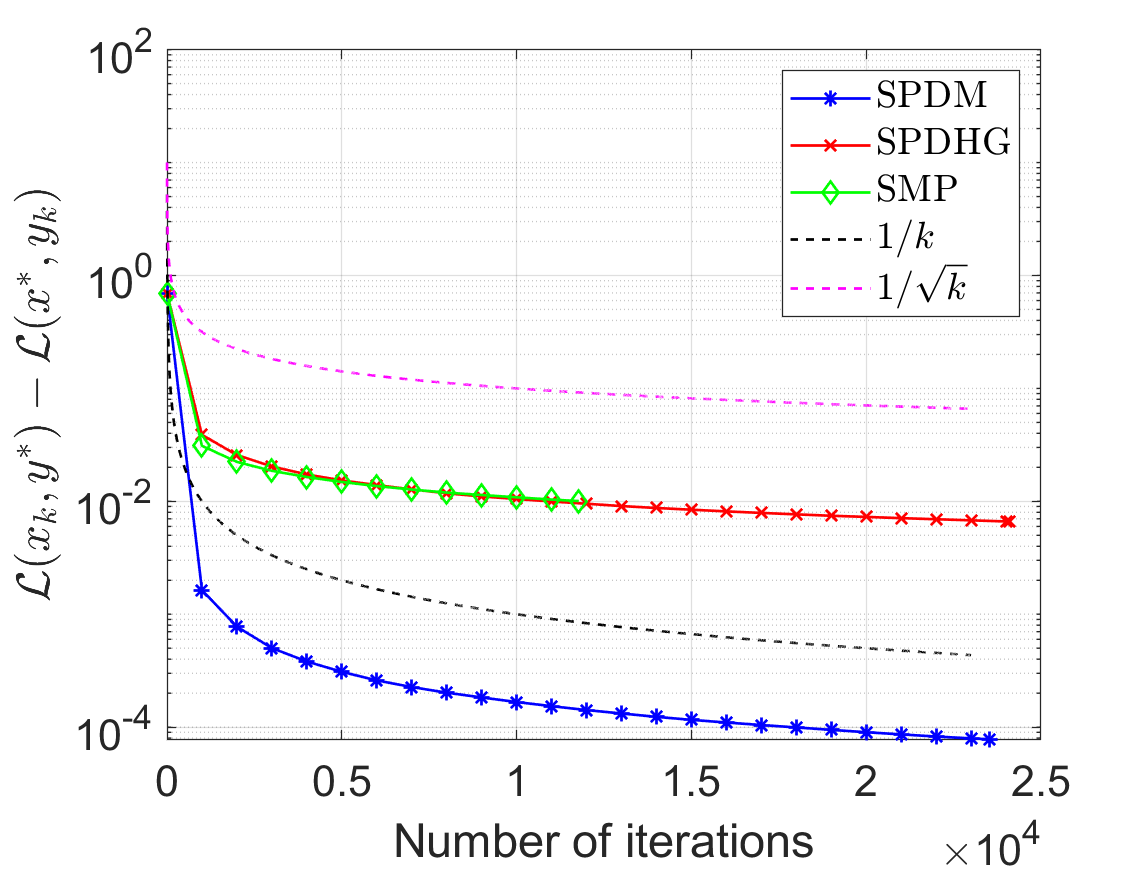} }}
    \subfloat[Colon-cancer]{{\includegraphics[width=3.9cm]{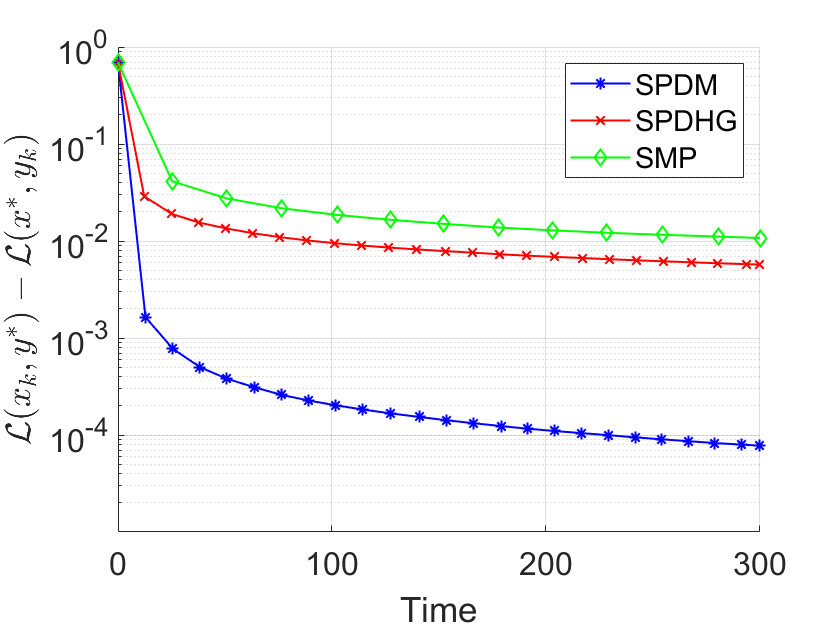} }}
     \subfloat[Leukemia]{{\includegraphics[width=3.9cm]{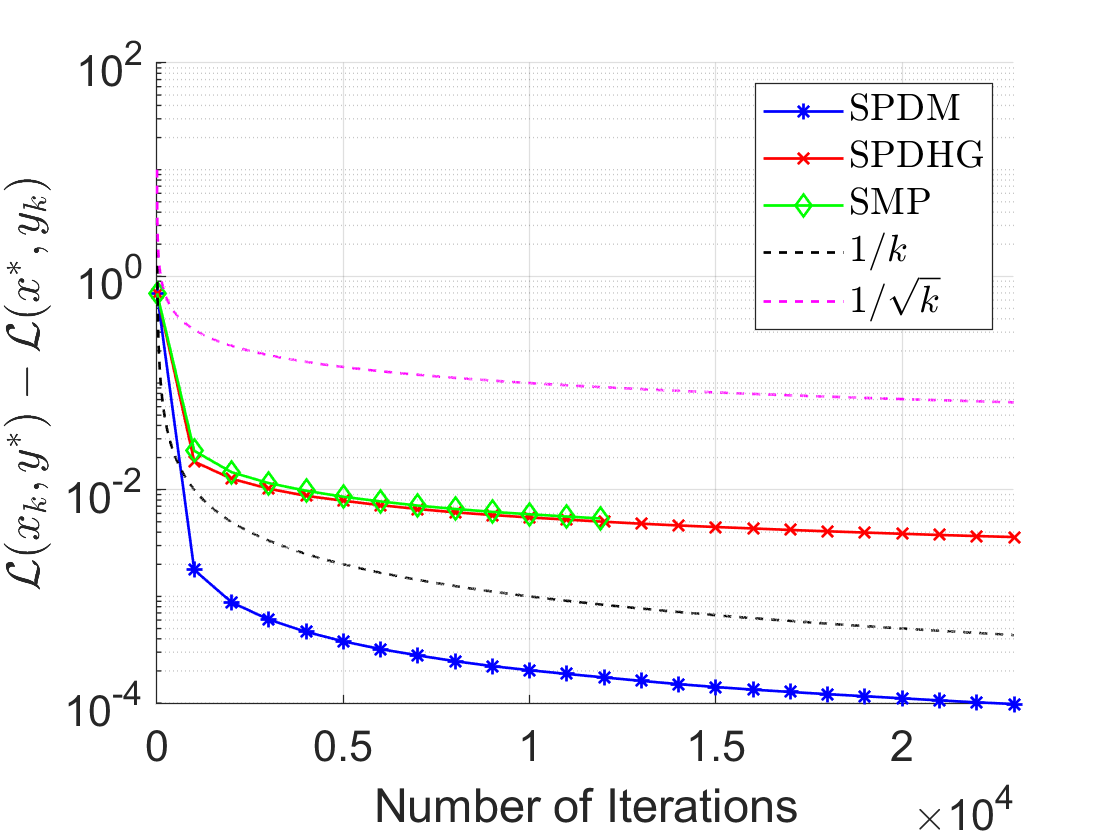} }}
    \subfloat[Leukemia]{{\includegraphics[width=3.9cm]{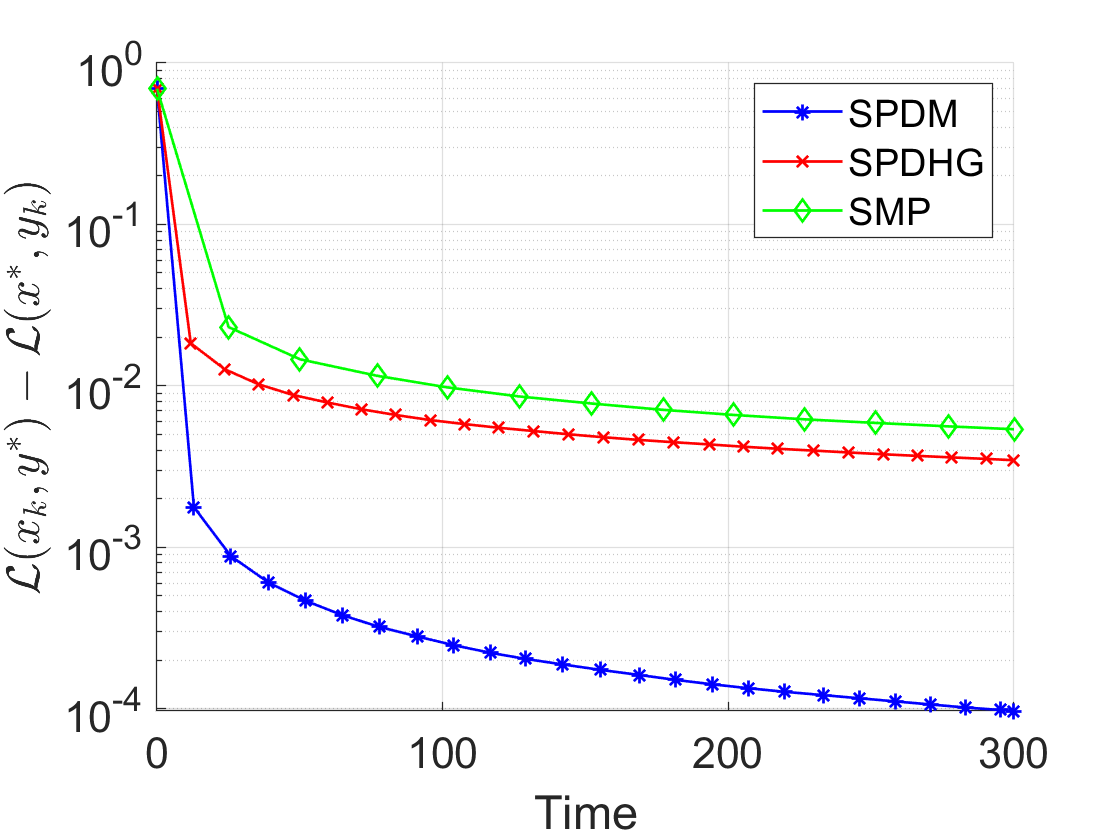} }}
   \caption{ Performance comparison of different methods in terms of iterations and running time for Colon-cancer ($62\times 2000$) and Leukemia ($38\times 7129$) data sets}
    \label{fig:dro1}\vspace{-3mm}
\end{figure}


     \begin{table}[htb]
\centering
\caption{Comparison of gap function 
$\sup_{(x,y)\in\mathcal Z}\left\{\Phi( x_T,y)-\Phi(x, y_T)\right\}$ for different methods }
\begin{tabular}{|c|c|c|c|}\hline
      &SPDM&SPDHG&SMP  \\ \hline
      Colon-cancer& 3.75e-4&1.51e-2&2.70e-2\\ \hline
      Leukemia&2.61e-4&6.99e-3&1.18e-2\\ \hline
    \end{tabular}
\label{tabel_dro}
\end{table}
\section{Concluding Remarks}\label{sec:conclude}
In this paper, we proposed an accelerated primal-dual scheme for solving a class of nonconvex-concave problems where the objective function satisfies the PL condition for both deterministic and stochastic settings. By combining an accelerated step in the minimization variable with an update involving a momentum in terms of the gradient of the objective function for the maximization variable, we obtained a convergence rate of \mbrs{$\mathcal O(\epsilon^{-4})$} and $\mathcal O(\epsilon^{-2})$  for the stochastic and deterministic problems, respectively. To the best of our knowledge, this is the first work that proposed a primal-dual scheme with momentum to solve PL-concave minimax problems. 

There are different interesting directions for future work: (i)  Investigating distributed variant of the proposed scheme over a network of agents; (ii) Considering a more general setting of nonconvex-concave SP problem and  developing a projection-free algorithm. 
\section*{APPENDIX}
In our analysis, we use the following technical lemma. 
\begin{lemma}\label{lem:error}
Given a arbitrary sequences $\{\bar\sigma_k\}_{k\geq 0}\subset \mathbb R^n$ and $\z{\{\bar\alpha_k\}}_{k\geq 0}\subset \mathbb R^{++}$, let $\{v_k\}_{k\geq0}$ be a sequence such that $v_0\in \mathbb R^n$ and $v_{k+1}=v_k+\z{\tfrac{\bar\sigma_k}{\bar \alpha_k}}$. Then, for all $k\geq 0$ and $x\in \mathbb R^n$,
$$\langle \sigma_k,x-v_k\rangle\leq {\z{\bar \alpha_k}\over 2}\|x-v_k\|^2-{\z{\bar \alpha_k}\over 2}\|x-v_{k+1}\|+{1\over 2 \z{\bar \alpha_k}}\|\bar\sigma_k\|^2.$$
\end{lemma}
To prove the convergence rate, we use the following lemma (proof is similar to Lemma 3 in \cite{ghadimi2016accelerated}).
\begin{lemma}\label{epsilon stationary}
For any given $z,\bar z\in \mathbb R^n$ and $c>0$, such that $\|z-\bar z\|\leq c \epsilon$, and $y\in \mathbb R^m$ let $\bar y\triangleq \mbox{prox}_{\sigma,h}(y+\sigma(\nabla_y\mathcal L(\bar z,y)+q+u))$ for some $q,u\in \mathbb R^m$ such that $\|q\|\leq \ell\|\nabla_x\mathcal L(x,y)\|$ and $\|u\|^2\leq \nu_y^2/b$ for some $\ell,\nu_y,b>0$. If $\|\nabla_x\mathcal L(z,y)\|^2+\|\bar y-y\|^2\leq \epsilon^2$, for some $\epsilon>0$, then $\|\nabla_x\mathcal L(z,y)\|\leq  \epsilon$ and $\nabla_y\mathcal L(z,y)\in h(\bar y)+\mathcal B(0,(1/\sigma+\ell+cL_{yx}) \epsilon+\nu_y/\sqrt b)$.
\end{lemma}
To facilitate the analysis, we define some notations.
\begin{definition}
\label{def}
Let $u_k^1\triangleq\nabla_y\ml_{\mathcal V_k}{(x_{k},y_{k})}-\nabla_y\ml{(x_{k},y_{k})}$, $u_k^2\triangleq\nabla_y\ml_{\mathcal V_k}{(x_{k-1},y_{k})}-\nabla_y\ml{(x_{k-1},y_{k})}$, and $u_k^3\triangleq\nabla_y \ml_{\mathcal V_k}{(z_{k+1},y_{k})}-\nabla_y \ml{(z_{k+1},y_{k})}$. Moreover, we define $\zeta_k\triangleq{L_{xx}\Gamma_{k}\over 2}\sum_{\tau=0}^{k}{(\gamma_\tau-\lambda_\tau)^2\over\Gamma_{\tau}\alpha_{\tau}}\|w_\tau\|^2$, $U_k\triangleq\langle \beta_ku_k^3+u_k^1-u_k^2,y_k+v_k\rangle$, and $\Xi_k\triangleq E_k^x+\bar E_k^x+\tfrac{L_{xx}\gamma_k^2}{2}\|w_k\|^2$, where $E_k^x\triangleq -\gamma_k\left\langle w_k,\nabla_x\ml(x_k,y_{k+1})\right\rangle\quad+L_{xx}\gamma_k^2\left\langle w_k,\nabla_x\ml(z_{k+1},y_{k+1})\right\rangle$, and $\bar E_k^x\triangleq \tfrac{L_{xx}\Gamma_{\zal{k-1}}(1-\alpha_k)^2}{ 2}\sum_{\tau=0}^{k}\tfrac{(\gamma_\tau-\lambda_\tau)^2}{\Gamma_{\tau}\alpha_{\tau}}, w_\tau^T\nabla_x\ml(z_{\tau+1},y_{\tau+1})$.
\end{definition}
In the next lemma, we provide a one-step analysis to obtain a bound for the norm of $\nabla_x\ml$ and progress of the dual iterates. This is the main building block of our convergence analysis in Theorem \ref{th1}.
\begin{lemma}\label{lemmath1}
Let $\{x_k, y_{k},z_k\}_{{k} \geq0}$ generated by Algorithm \ref{alg1} and suppose Assumptions \ref{assump0}-\ref{assump:stoch} hold.
 Moreover, let $\beta_k\triangleq\gamma_kC_k\mu$, $D\triangleq {\min \{{\tfrac{{\gamma_k C_k}}{4}, \tfrac{\beta_k}{{4\sigma_k}}}}\}$, $\sigma_k= \tfrac{\mu}{36L^2_{xy}}$ $\alpha_k=\tfrac{2}{k+1}$, $\lambda_k=\tfrac{1}{2L_{xx}}$ and $\gamma_k\in [\lambda_k,(1+\alpha_k/4)\lambda_k]$ for any $k\geq 0$ and $b=T$. Then, the following holds:
 \begin{align}\label{lemma_result}
  &\|\nabla_x \ml(z_{k^*},y_{k^*})\|^2+\|y_{k^*+1}-y_{k^*}\|^2\\
  &\leq \nonumber \tfrac{1}{TD}\biggr[ \ml(x_0,y^*)-\ml(x_T,y^*) +\tfrac{3\beta_0}{4\sigma_0}\|y^*-y_0\|^2 + \tfrac{{L^2_{xy}}\gamma_0^2}{2\beta_0\tau_0}\|\nabla_x \ml(z_0,y_0))\|^2\\\nonumber
  &\quad+\sum_{k=0}^{T-1}\left(\tfrac{1}{\bar \alpha_k} {\|\beta_k u_k^3+u_k^1-u_k^2\|^2}+\Xi_k+\zeta_k+U_k\right)\biggr].
\end{align}
\end{lemma}
\begin{proof}
Define $\Delta_k\triangleq\nabla_x \ml{(x_k,y_{k+1})}-\nabla_x \ml{(z_{k+1},y_{k+1})}$. From \mor{Assumption \ref{assump1}} and step \ref{update z} of Algorithm \zal{\ref{alg1}}, the following can be obtained,
\begin{align}\label{align5}
\nonumber\|\Delta_k\|&=\|\nabla_x \ml{(x_k,y_{k+1})}-\nabla_x \ml{(z_{k+1},y_{k+1})}\|\leq L_{xx}\|x_k-z_{k+1}\|\\
&=L_{xx}\|x_k-(1-\alpha_k)\tilde x_k-\alpha_k
x_k\|
=L_{xx}(1-\alpha_k)\|\tilde x_k-x_k\|.
\end{align}
\af{Define $w_k\triangleq\nabla_x\ml_{\mathcal U_k}(z_{k+1},y_{k+1})-\nabla_x\ml(z_{k+1},y_{k+1})$}, and using \eqref{assump3} and step \ref{update x} of Algorithm \zal{\ref{alg1}}, one can obtain 
\begin{align}\label {align6}
\nonumber &\ml{(x_{k+1},y_{k+1})}\\
\nonumber&\leq \ml{(x_k,y_{k+1})}+\langle\nabla_x \ml{(x_k,y_{k+1})},x_{k+1}-x_k\rangle+\tfrac{L_{xx}}{ 2}\|x_{k+1}-x_k\|^2\\
\nonumber&=\ml{(x_k,y_{k+1})}+\big\langle \Delta_k+\nabla_x \ml{(z_{k+1},y_{k+1})} -\gamma_k (\nabla_x  \ml{(z_{k+1},y_{k+1})}\af{+w_k})\big\rangle \\ \nonumber&\quad+L_{xx}\tfrac{\gamma_k^2}{ 2}\|\nabla_x\ml{(z_{k+1},y_{k+1})}\af{+w_k}\|^2\\
\nonumber&\leq \ml{(x_k,y_{k+1})}-\gamma_k \left( 1-\tfrac{L_{xx}\gamma_k}{ 2}\right)\|\nabla_x \ml{(z_{k+1},y_{k+1})}\|^2+\gamma_k\|\Delta_k\|\|\nabla_x\ml{(z_{k+1},y_{k+1})}\|
\\
 & \quad -\gamma_k\left\langle w_k,\nabla_x\ml(x_k,y_{k+1})\right\rangle+L_{xx}\gamma_k^2\left\langle w_k,\nabla_x\ml(z_{k+1},y_{k+1})\right\rangle+\tfrac{L_{xx}\gamma_k^2}{2}\|w_k\|^2.
\end{align}
Define 
$E_k^x\triangleq -\gamma_k\left\langle w_k,\nabla_x\ml(x_k,y_{k+1})\right\rangle+L_{xx}\gamma_k^2\left\langle w_k,\nabla_x\ml(z_{k+1},y_{k+1})\right\rangle.$
Combining  \eqref{align5} and \eqref{align6}:
\begin{align}\label{align3}
\nonumber\ml{(x_{k+1},y_{k+1})}&\leq \ml{(x_k,y_{k+1})}-\gamma_k\big( 1-\tfrac{ L_{xx}\gamma_k}{ 2}\big)\|\nabla_x \ml{(z_{k+1},y_{k+1})}\|^2\\\nonumber&\quad+L_{xx}(1-\alpha_k)\gamma_k\| \zal{\nabla_x}\ml{(z_{k+1},y_{k+1})}\|\|\tilde x_k-x_k\|  +E_k^x+\tfrac{L_{xx}\gamma_k^2}{2}\|w_k\|^2\\
\nonumber&\leq \ml{(x_k,y_{k+1})} -\gamma_k\big( 1-\tfrac{L_{xx}\gamma_k}{ 2}\big)\|\nabla_x \ml{(z_{k+1},y_{k+1})}\|^2\\
\nonumber&\quad+\tfrac{L_{xx}{\gamma_k}^2}{ 2}\|\zal{\nabla_x}\ml{(z_{k+1},y_{k+1})}\|^2+ \tfrac{L_{xx}(1-{\alpha_k})^2}{2}\|\tilde x_k-x_k\|^2+\af{E_k^x+\tfrac{L_{xx}\gamma_k^2}{2}\|w_k\|^2}\\
 \nonumber&=\ml{(x_k,y_{k+1})}-\gamma_k\big(1-{L_{xx}\gamma_k}\big)\|\nabla_x\ml{(z_{k+1},y_{k+1})}\|^2
 \\&\quad+\tfrac{L_{xx}(1-{\alpha_k})^2}{ 2}\|\tilde x_k-x_k\|^2 +E_k^x+\tfrac{L_{xx}\gamma_k^2}{2}\|w_k\|^2,
\end{align}
where we used $ab\leq\tfrac{(a^2+b^2)}{2}$. By steps \zal{\ref{update z}, \ref{update x} and \ref{update tilde x}} of Algorithm \zal{\ref{alg1}} one can obtain 
$\tilde x_{k+1}-x_{k+1}=(1-\alpha_k)\tilde x_k+\alpha_k x_k-\lambda_k(\nabla_x \ml{(z_{k+1},y_{k+1})}+w_k)-[x_k-\gamma_k(\nabla_x \ml{(z_{k+1},y_{k+1})}+w_k)]=(1-\alpha_k)(\tilde x_k-x_k)+(\gamma_k-\lambda_k)(\nabla_x \ml{(z_{k+1},y_{k+1})}+w_k).$
\zal{If we divide both sides of the above equality by $\Gamma_k$, summing over $k$ and using the definition of $\Gamma_k$, we obtain}
$\tilde x_{k+1}-x_{k+1}=\Gamma_k\sum_{\tau=0}^k\left({\gamma_\tau-\lambda_\tau \over \Gamma_\tau}\right)(\nabla_x \ml{(z_{\tau+1},y_{\tau+1})}+\af{w_\tau}).$
Using above equality, the Jensen's inequality, and the fact that
$\sum_{\tau=0}^k {\alpha_\tau \over \Gamma_\tau}={\alpha_0 \over \Gamma_0}+\sum_{\tau=1}^k{1\over  \Gamma_\tau }\left (1-{{\Gamma_\tau}\over{\Gamma_{\tau-1}}}\right)={1 \over \Gamma_0}+\left(\sum_{\tau=1}^k{1\over  \Gamma_\tau }-{1\over  \Gamma_{\tau-1}}\right)={1\over  \Gamma_k},$
we obtain
\begin{align}\label{align9}
\nonumber\| \tilde x_{k+1}-x_{k+1}\|^2&=\left \|\Gamma_k \sum_{\tau=0}^k\left(\tfrac{\gamma_\tau-\lambda_\tau}{\Gamma_\tau}\right)(\nabla_x \ml{(z_{\tau+1},y_{\tau+1})}+\af{w_\tau})\right\|^2\\
\nonumber &=\left\|\Gamma_k \sum_{\tau=0}^k\tfrac{\alpha_\tau}{\Gamma_\tau}\left[\left(\tfrac{\gamma_\tau-\lambda_\tau}{\alpha_\tau}\right)(\nabla_x \ml{(z_{\tau+1},y_{\tau+1})}+\af{w_\tau})\right]\right\|^2\\
\nonumber &\leq\Gamma_k \sum_{\tau=0}^k\tfrac{\alpha_\tau}{\Gamma_\tau}\left \|\left ( \tfrac{\gamma_\tau-\lambda_\tau}{\alpha_\tau}\right )(\nabla_x \ml{(z_{\tau+1},y_{\tau+1})}+\af{w_\tau}) \right \| ^2\\
&=\Gamma_k \sum_{\tau=0}^k\tfrac{(\gamma_\tau-\lambda_\tau)^2}{\Gamma_\tau \alpha_\tau}\|(\nabla_x \ml{(z_{\tau+1},y_{\tau+1})}+\af{w_\tau})\|^2.
\end{align}
Using \eqref{align9} in \eqref{align3}, one can obtain the following, 
\begin{align}\label{sum1}
\nonumber \ml{(x_{k+1},y_{k+1})}&\leq  \ml{(x_{k},y_{k+1})}-\gamma_k(1-L_{xx}\gamma_k)\|\nabla_x\ml{(z_{k+1},y_{k+1})}\|^2\\\nonumber&\quad+\af{E_k^x+\tfrac{L_{xx}\gamma_k^2}{2}\|w_k\|^2}+\tfrac{L_{xx}\Gamma_{\zal{k-1}}(1-\alpha_k)^2}{2}
\times \sum_{\tau=0}^\zal{k-1}\tfrac{(\gamma_\tau-\lambda_\tau)^2}{\Gamma_{\tau}\alpha_{\tau}}\|\nabla_x\ml{(z_{\tau+1},y_{\tau+1})}+w_\tau\|^2\\
\nonumber&\leq \ml{(x_{k},y_{k+1})}-\gamma_k(1-L_{xx}\gamma_k)\|\nabla_x\ml{(z_{k+1},y_{k+1})}\|^2+\af{E_k^x+\tfrac{L_{xx}\gamma_k^2}{2}\|w_k\|^2}\\\nonumber&\quad
+\tfrac{L_{xx}\Gamma_{k}}{ 2}\sum_{\tau=0}^{k}\tfrac{(\gamma_\tau-\lambda_\tau)^2}{\Gamma_{\tau}
\alpha_{\tau}}\|\nabla_x\ml{(z_{\tau+1},y_{\tau+1})}\|^2\\&\quad+\tfrac{L_{xx}\Gamma_{k}}{2}\sum_{\tau=0}^{k}\tfrac{(\gamma_\tau-\lambda_\tau)^2}{\Gamma_{\tau}\alpha_{\tau}}\|w_\tau\|^2+\tfrac{L_{xx}\Gamma_{\zal{k-1}}(1-\alpha_k)^2}{2}\times\af{\sum_{\tau=0}^{k}\tfrac{(\gamma_\tau-\lambda_\tau)^2}{\Gamma_{\tau}\alpha_{\tau}}}w_\tau^T\nabla_x\ml(z_{\tau+1},y_{\tau+1}).
\end{align}
Using Definition \ref{def}, summing both sides of \eqref{sum1} over $k$, and using the definition of $C_k$ in \eqref{align1}, we obtain the following
\begin{align}\label{align11}
\nonumber &\sum_{k=0}^{T-1}\left(\ml{(x_{k+1},y_{k+1})}-\ml{(x_{k},y_{k+1})}\right)\\ \nonumber&\leq-\sum_{k=0}^{T-1}\gamma_k(1-L_{xx}\gamma_k)\|\nabla_x \ml{(z_{k+1},y_{k+1})}\|^2+\sum_{k=0}^{T-1} \tfrac{L_{xx}\Gamma_\mor{{k}}}{2}\sum_{\tau=0}^\mor{{k}}\tfrac{(\gamma_\tau-\lambda_\tau)^2}{\Gamma_{\tau}\alpha_{\tau}}\|\nabla_x \ml{(z_{\zal{\tau+1}},y_{\zal{\tau+1}})}\|^2
\\ \nonumber&\quad+\mb{\sum_{k=0}^{T-1}(\Xi_k+\zeta_k)}\\
 &= \tfrac{L_{xx}}{ 2}\sum_{k=0}^{T-1}\tfrac{(\gamma_k-\lambda_k)^2}{\Gamma_{k}\alpha_{k}}\mor{\left(\sum_{\tau=k}^{T-1}\Gamma_\tau \right)} \|\nabla_x \ml{(z_{k+1},y_{k+1})}\|^2-\sum_{k=0}^{T-1}\gamma_k C_k\|\nabla_x \ml{(z_{k+1},y_{k+1})}\|^2+\mb{\sum_{k=0}^{T-1}(\Xi_k+\zeta_k)}.
\end{align}
From \eqref{align11} and Definition \eqref{def PL}, one can obtain
\begin{align*}
 &\sum_{k=0}^{T-1}(\ml(x_{k+1},y_{k+1})-\ml(x_{k},y_{k+1}))\\ \nonumber&\leq-\sum_{k=0}^{T-1}\gamma_k C_k\mu(\ml{(z_{k+1},y_{k+1})}-\ml(x^\ast\zal{(y_{k+1})},y_{k+1}))
-\sum_{k=0}^{T-1}{\gamma_k C_k\over 2}\|\nabla_x \ml{(z_{k+1},y_{k+1})}\|^2\\ \nonumber&\quad+\mb{{\sum_{k=0}^{T-1}(\Xi_k+\zeta_k)}.}
\end{align*}
Adding $\sum_{k=0}^{T-1}(\ml(x_{k+1},y)-\ml(x_{k},y))$ to both sides:
\begin{align*}
\nonumber &\sum_{k=0}^{T-1}(\ml(x_{k+1},y)-\ml(x_{k},y))\\
\nonumber&\leq
 \sum_{k=0}^{T-1}{(\ml(x_{k+1},y)-\ml(x_{k+1},y_{k+1}))}
+ \sum_{k=0}^{T-1}{(\ml(x_{k},y_{k+1})-\ml(x_{k},y))}\\ \nonumber&\quad- \sum_{k=0}^{T-1}\zal{\gamma_k C_k\mu}{(\ml(z_{k+1},y_{k+1})-\ml(z_{k+1},y)\zal{)}}-\sum_{k=0}^{T-1}\gamma_k C_k\mu (\ml(z_{k+1},y)-\ml(x^*(y_{k+1}),y_{k+1}))\\
\nonumber &\quad-\zal{\sum_{k=0}^{T-1}}{\gamma_k C_k\over 2}\|\nabla_x \ml{(z_{k+1},y_{k+1})}\|^2+\mb{\sum_{k=0}^{T-1}(\Xi_k+\zeta_k)}.
\end{align*}
Using concavity of $\ml$ over $y$, one can obtain
\begin{align}\label{align14}
\nonumber &\sum_{k=0}^{T-1}(\ml(x_{k+1},y)-\ml(x_{k},y))\\
\nonumber&\leq -\sum_{k=0}^{T-1}\gamma_k C_k\mu(\ml{(z_{k+1},y)}-\ml({x^\ast\zal{(y_{k+1})},y_{k+1}})\zal{)}
-\sum_{k=0}^{T-1}{\gamma_k C_k\over 2}\|\nabla_x \ml{(z_{k+1},y_{k+1})}\|^2\\ \nonumber&\quad+\sum_{k=0}^{T-1}[\langle\nabla_y\ml{(x_{k+1},y_{k+1})}+\gamma_kC_k\mu\nabla_y\ml{(z_{k+1},y_{k+1})},y-y_{k+1}\rangle\\&\quad
+\langle
\nabla_y\ml{(x_{k},y_{k+1})},y_{k+1}-y\rangle
]+\mb{\sum_{k=0}^{T-1}(\Xi_k+\zeta_k)}.
\end{align}
Let us define $u_k^1=\nabla_y\ml_{\mathcal V_k}{(x_{k},y_{k})}-\nabla_y\ml{(x_{k},y_{k})}$, $u_k^2=\nabla_y\ml_{\mathcal V_k}{(x_{k-1},y_{k})}-\nabla_y\ml{(x_{k-1},y_{k})}$, $u_k^3=\nabla_y \ml_{\mathcal V_k}{(z_{k+1},y_{k})}-\nabla_y \ml{(z_{k+1},y_{k})}$, define $\bar p_k=\nabla_y \ml{(z_{k+1},y_{k})}$, and $\bar q_k={1\over \beta_k}(\nabla_y\ml{(x_{k},y_{k})}-\nabla_y\ml{(x_{k-1},y_{k})})$. \mor{From optimality condition of step \ref{update y} in Algorithm \ref{alg1},} letting $s_{k}=\bar p_{k}+\bar q_{k}+u_k^1+u_k^2+u_k^3$, one can obtain 
$h(y_{k+1})-\langle  s_{k},y_{k+1}-y\rangle \leq
h(y)+\tfrac{1}{2\sigma_k}[\|y-y_k\|^2-\|y-y_{k+1}\|^2-\|y_{k+1}-y_{k}\|^2].$
Multiplying both sides by $\beta_k=\gamma_kC_k\mu$ and summing over $k$, we obtain,
\begin{align}\label{bound9}
&\sum_{k=0}^{T-1}\beta_k(h(y_{k+1})-\langle  s_{k},y_{k+1}-y\rangle) \leq
\sum_{k=0}^{T-1}\beta_k(h(y)+\tfrac{1}{2\sigma_k}[\|y-y_k\|^2-\|y-y_{k+1}\|^2-\|y_{k+1}-y_{k}\|^2]).
\end{align}
Now, we simplify the inner products involving in \eqref{align14} and \eqref{bound9} using the definition of $\bar p_k$ and $\bar q_k$.
\begin{align}\label{simply}
  \nonumber& \sum_{k=0}^{T-1}\langle\nabla_y\ml{(x_{k+1},y_{k+1})}+\beta_k\nabla_y\ml{(z_{k+1},y_{k+1})},y-y_{k+1}\rangle+\langle
\nabla_y\ml{(x_{k},y_{k+1})},y_{k+1}-y\rangle+\langle  s_{k},y_{k+1}-y\rangle)\\
&=\sum_{k=0}^{T-1}[\beta_{k+1}\langle \bar q_{k+1},y-y_{k+1}\rangle-\beta_k\langle \bar q_k,y-y_k\rangle+\langle \bar q_k, y_{k+1}-y_k\rangle].
\end{align}
\mor{\zal{Moreover, using Young's inequality, and step \ref{update x} in Algorithm \ref{alg1}}, one can obtain}
\begin{align}\label{align15}
&\langle \bar q_k, y_{k+1}-y_k\rangle\leq \tfrac{L_{xy}}{ \beta_k}\|x_k-x_{k-1}\|\|y_{k+1}-y_k\|\leq \tfrac{L_{xy}^2\gamma_{k-1}^2}{2\beta_k\tau_k}\|\nabla_x\ml{(z_k,y_k)}\|^2+\tfrac{\tau_k}{2}\|y_{k+1}-y_k\|^2.
\end{align}
Summing \eqref{bound9} and \eqref{align14}, using \eqref{align15} and \eqref{simply}, we get,
\begin{align}\label{align22}
&\nonumber \sum_{k=0}^{T-1}(\ml(x_{k+1},y)-\ml(x_{k},y))+ \sum_{k=0}^{T-1}\beta_k(\ml(z_{k+1},y)-\ml(x^\ast\zal{(y_{k+1})},y_{k+1}))+\sum_{k=0}^{T-1}\beta_k(h(y_{k+1})-h(y))\\ \nonumber
&\leq \sum_{k=0}^{T-1}{(\tfrac{L_{xy}^2\gamma_{k-1}^2}{2{\beta_k}\tau_k}-\tfrac{\gamma_k C_k}{2})}\|\nabla_x\ml{(z_k,y_k)}\|^2+\sum_{k=0}^{T-1}[\beta_{k+1}\langle \bar q_{k+1},y-y_{k+1}\rangle-\beta_k\langle \bar q_k,y-y_k\rangle]\\
\nonumber& \quad+ \sum_{k=0}^{T-1}\tfrac{\beta_k}{ 2\sigma_k}[\|y-y_k\|^2-\|y-y_{k+1}\|^2]+\sum_{k=0}^{T-1}{(\tfrac{\az{\tau_k}}{2}-\tfrac{\beta_k}{{2\sigma_k}})}\|y_{k+1}-y_k\|^2+\tfrac{L_{xy}^2\gamma_0^2}{ 2\beta_0\tau_0}\|\nabla_x\ml{(z_0,y_0)}\|^2\\
&\quad-\tfrac{\gamma_{T-1}C_{T-1}}{ 2}\|\nabla_x \ml (z_{T},y_T)\|^2+\af{\sum_{k=0}^{T-1}\langle\beta_ku_k^3+u_k^1-u_k^2,y_{k+1}-y\rangle}+\mb{\sum_{k=0}^{T-1}(\Xi_k+\zeta_k)},
\end{align}
where $\gamma_{-1} = \gamma_0$. From Cauchy-Schwartz inequality, using Lemma \ref{lem:error} where we choose $v_0=y_0$, and defining $U_k\triangleq\langle \beta_ku_k^3+u_k^1-u_k^2,y_k+v_k\rangle$, the following holds
\begin{align}\label{2align22}
  \nonumber &\langle\beta_ku_k^3+u_k^1-u_k^2,y_{k+1}-y\pm y_k\rangle\leq \tfrac{1}{ 2\bar \alpha_k}\|\beta_ku_k^3+u_k^1-u_k^2\|^2+\tfrac{\bar \alpha_k}{ 2}\|y_{k+1}-y_k\|^2\\ \nonumber&\quad +\langle \beta_ku_k^3+u_k^1-u_k^2,y_{k}-y\pm v_k\rangle\\
   &\leq \tfrac{1}{ \bar \alpha_k}\|\beta_ku_k^3+u_k^1-u_k^2\|^2+\tfrac{\bar \alpha_k}{ 2}\|y_{k+1}-y_k\|^2+\tfrac{\bar \alpha_k}{ 2}\|y-v_k\|^2-\tfrac{\bar \alpha_k}{ 2}\|y-v_{k+1}\|^2+U_k,
\end{align}
for some $\bar \alpha_k\geq 0$. Hence, using \eqref{2align22} in \eqref{align22} and rearranging terms, one can obtain the following, 
\begin{align}\label{3align23}
  \nonumber &\quad -\sum_{k=0}^{T-1}\underbrace{(\tfrac{L_{xy}^2{\gamma^2_{k-1}}}{2{\beta_k}\tau_k}-\tfrac{\gamma_k C_k}{2})}_{\text{term (A)}}\|\nabla_x\ml{(z_k,y_k)}\|^2-\sum_{k=0}^{T-1}\underbrace{(\tfrac{\az{\tau_k}}{2}-\tfrac{\beta_k}{{2\sigma_k}}+\tfrac{\bar \alpha_k}{ 2})}_{\text{term (B)}}\|y_{k+1}-y_k\|^2\\
\nonumber&\leq -\sum_{k=0}^{T-1}(\ml(x_{k+1},y)-\ml(x_{k},y))- \sum_{k=0}^{T-1}\beta_k(\ml(z_{k+1},y)-\ml(x^\ast\zal{(y_{k+1})},y_{k+1}))- \sum_{k=0}^{T-1}\beta_k(h(y_{k+1})-h(y))\\
\nonumber&\quad +\sum_{k=0}^{T-1}[\beta_{k+1}\langle \bar q_{k+1},y-y_{k+1}\rangle-\beta_k\langle \bar q_k,y-y_k\rangle] +\sum_{k=0}^{T-1}\tfrac{\beta_k}{ 2\sigma_k}[\|y-y_k\|^2-\|y-y_{k+1}\|^2\\
\nonumber&\quad+\tfrac{1}{ 2}\|y-v_k\|^2-\frac{1}{2}\|y-v_{k+1}\|^2]+\tfrac{L_{xy}^2\gamma_0^2}{ 2\beta_0\tau_0}\|\nabla_x\ml{(z_0,y_0)}\|^2-\tfrac{\gamma_{T-1}C_{T-1}}{ 2}\|\nabla_x \ml (z_{T},y_T)\|^2\\
&\quad+\af{\sum_{k=0}^{T-1}\tfrac{1}{ \bar\alpha_k}\|\beta_ku_k^3+u_k^1-u_k^2\|^2}+\mb{\sum_{k=0}^{T-1}(\Xi_k+\zeta_k+U_k)}.
\end{align}
Choosing the parameters such that \az{$\sigma_k \leq\tfrac{\mu^2}{216L^2_{xy}}$}, $\alpha_k=\tfrac{2}{k+1}$, $\lambda_k=\tfrac{1}{2L_{xx}}$, $\tau_k=\tfrac{9L^2_{xy}}{\mu}$ , $\bar\alpha_k=\tfrac{\beta_k}{4\sigma_k}$, and $\gamma_k\in [\lambda_k,(1+\alpha_k/4)\lambda_k]$ for any $k\geq 0$, one can show that in \eqref{3align23} Term (A)$\leq{-{\gamma_k C_k}\over4}$ and Term (B)$\leq {-\beta_k\over{4\sigma_k}}$. Therefore, choosing $k^*=\mbox{argmin}\{\|\nabla \ml (z_{k},y_{k})\|^2+\|y_{k+1}-y_{k}\|^2\}$, the left hand side (LHS) of \eqref{3align23} can be bounded from below by $\big(\sum_{k=0}^{T-1} \min \{{\tfrac{{\gamma_k C_k}}{4}, \tfrac{\beta_k}{{4\sigma_k}}}\}\big)\big(\|\nabla_x \ml(z_{k^*}),y_{k^*})\|^2+\|y_{k^*+1}-y_{k^*}\|^2\big)$.
Moreover, letting $(x^*,y^*)$ to be an arbitrary saddle point solution of \eqref{main}, choosing $y=y^*$, using the fact that $\ml(x^\ast{(y_{k+1})},y_{k+1}) \leq \ml(x^*,y_{k+1})$ and \eqref{align15}, one can obtain:
\begin{align}
  \nonumber\|\nabla_x \ml(z_{k^*},y_{k^*})\|^2+\|y_{k^*+1}-y_{k^*}\|^2&\leq \nonumber \tfrac{1}{TD}\biggr[ \ml(x_0,y^*)-\ml(x_T,y^*) +\tfrac{3\beta_0}{4\sigma_0}\|y^*-y_0\|^2\\\nonumber
  &\quad + \tfrac{{L^2_{xy}}\gamma_0^2}{2\beta_0\tau_0}\|\nabla_x \ml(z_0,y_0))\|^2\\\nonumber
  &\quad+\sum_{k=0}^{T-1}\tfrac{1}{\bar \alpha_k} \underbrace{\|\beta_k u_k^3+u_k^1-u_k^2\|^2}_{\text{term (C)}}+\mb{\sum_{k=0}^{T-1}(\Xi_k+\zeta_k+U_k)}\biggr],
\end{align}
where $D\triangleq {\min \{{\tfrac{{\gamma_k C_k}}{4}, \tfrac{\beta_k}{{4\sigma_k}}}}\}$ and we used ${\sum_{k=0}^{T-1}D}=TD$. 
\end{proof}
Now, we are ready to prove Theorem \ref{th1} and establish the convergence rate results.

{\bf Proof of Theorem \ref{th1}.} From \eqref{lemma_result}, we have that
\begin{align}
  \nonumber\|\nabla_x \ml(z_{k^*},y_{k^*})\|^2+\|y_{k^*+1}-y_{k^*}\|^2&\leq \nonumber \tfrac{1}{TD}\biggr[ \ml(x_0,y^*)-\ml(x_T,y^*) +\tfrac{3\beta_0}{4\sigma_0}\|y^*-y_0\|^2\\\nonumber&\quad + \sum_{k=0}^{T-1}\biggr(\tfrac{1}{\bar \alpha_k} \underbrace{\|\beta_k u_k^3+u_k^1-u_k^2\|^2}_{\text{term (C)}}+\Xi_k+\zeta_k+U_k\biggr)\biggr]  \nonumber \\
  &\quad + \tfrac{{L^2_{xy}}\gamma_0^2}{2\beta_0\tau_0}\|\nabla_x \ml(z_0,y_0))\|^2
\end{align}
{ Taking conditional expectation, one can show that $\mathbb E[C \mid \mathcal H_k ]\leq \tfrac{9\nu^2_y}{T}, \mathbb E[\Xi_k \mid \mathcal F_k]=\tfrac{L_{xx} \gamma^2_k \nu^2_x}{2T}$, $\mathbb E[\zeta_k \mid \mathcal F_k]\leq \tfrac{L_{xx} \lambda^2_k \nu^2_x}{32T},$ and $\mathbb E[U_k \mid \mathcal H_k ]=0$. Hence, we obtain:
\begin{align}\label{last inequal}
    \nonumber\mathbb E \biggr[ \|\nabla_x \ml(z_{k^*},y_{k^*})\|^2+\|y_{k^*+1}-y_{k^*}
    \|^2\biggr]&\leq 
  \tfrac{1}{TD} \biggr[ \ml(x_0,y^*)-\ml(x_T,y^*) +\tfrac{3\beta_0}{4\sigma_0}\|y^*-y_0\|^2 \nonumber \\
  & \quad + \tfrac{{L^2_{xy}}\gamma_0^2}{2\beta_0\tau_0}\|\nabla_x \ml(z_0,y_0))\|^2 \nonumber \\&\quad +\sum_{k=0}^{T-1}\left( \tfrac{9\nu^2_y}{\bar \alpha_k T}+\tfrac{L_{xx}, \gamma^2_k \nu^2_x}{2T}+\tfrac{L_{xx} \lambda^2_k \nu^2_x}{32T}\right)\biggr]\leq \mathcal O(1/T).
\end{align}
Moreover, from the steps of Algorithm \ref{alg1}, $\|z_{k+1}-z_k\|\leq \lambda_{k-1}\|r_{k-1}\|+\|x_k-\tilde x_k\|$. Using steps \ref{update x} and \ref{update tilde x}, one can show that $\|x_k-\tilde x_k\|\leq \mathcal O(\tfrac{1}{(T+1)\sqrt T})$. Hence $\|z_{k+1}-z_k\|\leq \mathcal O(\sqrt\epsilon)$. Invoking Lemma \ref{epsilon stationary}, we conclude that $(z_{k^*},y_{k^*})$ is an $\epsilon$-stationary point of problem \eqref{main}.} 
To achieve an $\epsilon$-stationary point, we let the rhs of \eqref{last inequal} equal to $\epsilon^2$ which implies that $T=\mathcal O(\epsilon^{-2})$. Hence, total number of sample gradient evaluations is $\sum_{k=0}^{T-1}b=T^2=\mathcal O(\epsilon^{-4})$, since we chose $b=T$.   \qed
\bibliographystyle{alpha}
\bibliography{IEEEexample}
\end{document}